\documentclass[a4paper,12pt]{article}
\usepackage{amssymb,amsmath,amsthm,latexsym}
\usepackage{amsfonts}
\usepackage{amsfonts}
\usepackage{graphicx}
\usepackage[pdftex,bookmarks,colorlinks=false]{hyperref}
\usepackage{verbatim}
\usepackage[font=small,labelfont=bf]{caption}
\usepackage{subcaption}
\usepackage[hmargin=1.2in,vmargin=1.2in]{geometry}
\usepackage{authblk}
\usepackage{multirow}
\newtheorem{theorem}{Theorem}[section]

\newtheorem{corollary}[theorem] {Corollary}
\newtheorem{definition}[theorem]{Definition}

\newtheorem{problem}[theorem]{Problem}
\newtheorem{proposition}[theorem]{Proposition}

\title{\bf A study on the curling number of certain graph classes}
\author{Susanth C}
\affil{\small Department of Mathematics\\ Research \& Development Centre\\ Bharathiar University\\ Coimbatore - 641046, Tamilnadu, India.\\ email:{\em susanth\_c@yahoo.com}} 
  
\author{Sunny Joseph Kalayathankal}
\affil{\small Department of Mathematics\\ Kuriakose Elias College\\ Kottayam - 686561, Kerala, India.\\ email:{\em sunnyjoseph2014@yahoo.com}}

\author{Naduvath Sudev}
\affil{\small Department of Mathematics\\ Vidya Academy of Science \& Technology\\  Thrissur - 680501, Kerala, India.\\ email:{\em sudevnk@gmail.com}}

\author{Kaithavalappil Chithra}
\affil{\small Naduvath Mana, Nandikkara\\ Thrissur, India.\\ email:{\em chithrasudev@gmail.com}}

\author{Johan Kok}
\affil{\small Tshwane Metro Police Department\\ City of Tshwane, South Africa.\\ email:{\em kokkiek2@tshwane.gov.za}}

\date{}

\begin{document}
\maketitle
\newpage
\begin{abstract}
Given a finite nonempty sequence $S$ of integers, write it as $XY^k$, consisting of a prefix $X$ (which may possibly be empty), followed by $k$ copies of a non-empty string $Y$. Then, the greatest such integer $k$ is called the curling number of $S$ and is denoted by $cn(S)$. The notion of curling number of graphs has been introduced in terms of their degree sequences, analogous to the curling number of integer sequences.  In this paper, we study the curling number of certain graph classes and graphs associated to given graph classes.
\end{abstract}


\section{Introduction}
For terms and definitions in graph theory, see \cite{BM1,CGT,FH,DBW,RJW} and for more about different graph classes, refer to \cite{AVJ, JAG}.  Unlesse mentioned otherwise, the graphs considered in this paper are simple, finite, connected and undirected.  The notion of curling number of integer sequences is introduced in \cite{BJJA} as follows. 

\begin{definition}{\rm
\cite{BJJA} Let $S=S_1S_2S_3\ldots S_n$ be a finite string. Write $S$ in the form $XYY\ldots Y=XY^k$, consisting of a prefix $X$ (which may be empty), followed by $k$ copies of a non-empty string $Y$. This can be done in several ways. Pick one with the greatest value of $k$ . Then, this integer $k$ is called the curling number of $S$ and is denoted by $cn(S)$.}
\end{definition}

\begin{definition}{\rm
The \textit{Curling Number Conjecture} (see \cite{BJJA}) states that if one starts with any finite string, over any alphabet, and repeatedly extends it by appending the curling number of the current string, then eventually one must reach a 1.}
\end{definition}

The concept of curling number of integer sequences has been extended to the degree sequences of graphs in \cite{KSC} and the corresponding properties and characteristics of certain standard graphs have been studied in \cite{KSC, SS1,SS2}.  

\begin{definition}{\rm
\cite{KSC} A maximal degree subsequence with equal entries is called an \textit{identity subsequence}.  An identity subsequence can be a curling subsequence and the number of identity curling subsequences found in a simple connected graph $G$ is denoted $ic(G)$.}
\end{definition}

The following is an important and relevant result on curling numbers of graphs, which is relevant in our present study.  


\begin{theorem}
\cite{KSC} For the degree sequence of a non-trivial, connected graph $G$ on $n$ vertices, the curling number conjecture holds.
\end{theorem}

\begin{definition}{\rm 
\cite{KSC} A \textit{curling subsequence} of a simple connected graph $G$ is defined to be a maximal subsequence $C$ of the well-arranged degree sequence of $G$ such that $cn(C)=max\{cn(S_0)\}$ for all possible initial subsequences $S_0$.  The \textit{curling number} of a graph $G$ is the curling number of a curling subsequence $C$ of $G$.  That is, $cn(G)=cn(C)$, where $C$ is a curling subsequence of $G$.  Note that a graph $G$ can have a number of curling subsequences.}
\end{definition}

\noindent The following theorem is an important result on the curling number of a given graph.

\begin{theorem}
\cite{KSC} If a graph $G$ is the union of $m$ simple connected graphs $G_i;1\leq i\leq m$ and the respective degree sequences are re-arranged as strings of identity subsequences, then
\[ cn(G) = \begin{cases}
 \max\{cn(G_i)\} & \text{; if $X_i,X_j$ are not pairwise similar},\\
        \max $$\sum_{i=1}^{m} k_i $$ & \text{; for all integer of similar identify subsequences}.
\end{cases}\] 
\end{theorem}

As we have already seen, the degree sequence of an arbitrary graph $G$ can be written as a string of identity curling subsequences and hence the notion of compound curling number of a graph $G$ has been introduced in \cite{KSC} as given below.

\begin{definition}{\rm
\cite{KSC} Let the degree sequence of a graph $G$ be written as a string of identity curling subsequences, say $X_1^{k_1} \circ X_2^{k_2} \circ X_3^{k_3} \ldots \circ X_l^{k_l}$.  The \textit{compound curling number} of $G$, denoted by $cn^c(G)$, is defined to be $cn^c(G)=\prod_{i=1}^{l} k_i.$}
\end{definition}

The curling number and compound curling number of certain fundamental standard graphs have been determined in \cite{KSC}. 


\begin{proposition}
\cite{KSC} The compound curling number of any regular graph $G$ is equal to its curling number.
\end{proposition}


In this paper, we extend these studies on curling number of graphs to the associated graphs of certain graph classes. By the size of a sequence, we mean the number of elements in that sequence.

\subsection{Curling number of complements of graphs}

The most interesting question that arises when we study about the curling number of certain graphs associated with given graphs is about the curling number of the complements of the given graphs. The following proposition discusses the curling number of the complements of given graphs. 
 
\begin{proposition}
For any graph $G, cn(\bar{G})=cn(G)$ and $cn^c(\bar{G})=cn^c(G)$.
\end{proposition}
\begin{proof}
Assume $d_1^{n_1}\circ d_2^{n_2} \circ \ldots \circ d_r^{n_r}$ be the degree sequence of a graph $G$.  Without loss of generality, $n_r=max(n_i, i=1 \ldots r)$.  

Now we have $G\cup\bar{G}=K_n$. Hence $d_G(V)+d_{\bar{G}}(V)=n-1$.  Therefore, $d_{\bar{G}}(V)=(n-1)-d_G(V)$. Hence, the degree sequence of $\bar{G}$ will be $(n-1-d_1)^{n_1} \circ (n-1-d_2)^{n_2} \circ \ldots \circ (n-1-d_r)^{n_r}$.  Since $n_r=\{n_i\}, 1\leq i \leq r$, we have $cn(\bar{G})=n_r = cn(G)$.  Also, $cn^c(\bar{G})=\prod_{i=1}^{r}n_i=cn^c(G)$. 
\end{proof}

\subsection{Curling number of line graphs}
Another well known graph that is associated with a given graph is its line graph. The curling number of a line graph of a regular graph is determined in the following result.

\begin{proposition}
For a regular graph $G(V,E)$, the curling number of the line graph $L(G)=|E|$.
\end{proposition}
\begin{proof}
Let $G$ be an $r$-regular graph on $n$ vertices.  Then $L(G)$ is a $2r-2$ regular graph on $|E|$ vertices. Since $L(G)$ is also regular, we have $cn(L(G))=|V(L(G))|=|E|$.
\end{proof}

The bounds for the sum of curling numbers of a regular graph and its line graph is determined in the following proposition.

\begin{proposition}
If $G(V,E)$ is a regular graph then $2|V|-1\leq cn(G)+cn(L(G))\leq \frac{|V|(|V|+1)}{2}$.
\end{proposition}
\begin{proof}
Let $G$ be an regular graph on $n$ vertices.  Then $L(G)$ is also regular on $|E|$ vertices.  By the above preposition, we have $cn(L(G))=|E|$.  
Therefore, we have 
\begin{equation}\label{eq1}
cn(G)+cn(L(G))=|V|+|E|
\end{equation}
The minimum connected graph is a tree, for which $|E|=|V|-1$.  In this case, $|V|+|E|=2|V|-1$.  If $n=2$, we have $G\cong K_2, L(G)\cong K_1$. Also $cn(G)=cn(K_2)=2, cn(L(G))=cn(K_1)=1$.  Therefore, $cn(G)+cn(L(G))=|V|+|E|=3=2|V|-1$. Now note that any tree with $n>2$ is not a regular graph and hence we have
\begin{equation}\label{eq2}
2|V|-1<cn(G)+cn(L(G))
\end{equation}
If $G$ is a maximal connected then $G\cong K_n$.  Here $L(G)$ is a $2|V|-4$ regular graph and hence $cn(L(G))=|V(L(G))=\frac{|V|[|V|-1]}{2}$.
Hence, $cn(G)+cn(L(G))=|V|+|E|=|V|+\frac{|V|[|V|-1]}{2}=\frac{|V|[|V|+1]}{2}$.
Therefore, $2|V|-1\leq cn(G)+cn(L(G))\leq \frac{|V|(|V|+1)}{2}$. 
\end{proof}

For $n\geq 3$, a \textit{wheel graph} $W_{n+1}$ is the graph $K_1+ C_{n}$ (see \cite{FH}).  A wheel graph $W_{n+1}$ has $n+1$ vertices and $2n$ edges. 

The curling number of a wheel graph is determined in \cite{KSC}. Now we determine the curling number and compound curling number of the line graph of a wheel graph. 

\begin{proposition}
For a wheel graph $W_{n+1}=C_n+K_1$, we have $cn(L(W_{n+1})=cn(W_{n+1}))$ and $cn^c(L(W_{n+1})=(cn(W_{n+1}))^2$.
\end{proposition}
\begin{proof}
Let $G\cong W_{n+1}=C_n+K_1$.  Let $L=L(G)$.  The vertices in $L$ corresponding to the spokes (the edges connecting the internal vertex and the vertices in the outer cycle) of $G$ induce a clique $K_n$ in $L$.  Moreover the vertices in $L$ corresponding to the edges in the outer cycle $C_n$ in $G$ have degree $4$ in $L$, since for any edge $v_iv_{i+1}$ in $C_n$, there will be two more adjacent edges to both $v_i$ and $v_{i+1}$ in $G$.  Therefore, the degree sequence of $L$ is given by $S=(n-1)^n \circ (4)^n$.  Therefore, $cn(L)=n=cn(G)$ and $cn^c(L)=n^2=(cn(G))^2$. 
\end{proof}

Another similar graph whose curling number has been determined in \cite{KSC} is a \textit {helm graph} (\cite{JAG}) which is defined to be a graph obtained from a wheel by attaching one pendant edge to each vertex of the cycle. In the following result, we determine the curling number and compound curling number of a helm graph. 

\begin{corollary}
For a helm graph $H_n$, we have $cn(L(H_n))=n$ and $cn^c(L(H_n))=cn(H_n)^3$.
\end{corollary}
\begin{proof}
Let $H_n$ be a helm graph on $2n+1$ vertices and $3n$ edges. The vertices in $L(H_n)$ corresponding to the spokes (the edges connecting the internal vertex and the vertices in the outer cycle) of $H_n$ induce a clique $K_n$ in $L(H_n)$. Moreover the vertices in $L(H_n)$ corresponding to the edges in the outer cycle $C_n$ in $H_n$ have degree $6$ in $L(H_n)$, since for any edge $v_iv_{i+1}$ in $C_n$, there will be three more adjacent edges to both $v_i$ and $v_{i+1}$ in $H_n$. Therefore, the degree sequence of $L(H_n)$ is given by $S=(n-1)^n \circ (6)^n\circ (3)^n$.  Therefore, $cn(L(H_n))= cn(H_n)=n$ and $cn^c(L(H_n))=n^3=cn((H_n)^3)$. 
\end{proof}

\subsection{Curling number of subdivision of a graph}

First, recall the notion of the subdivision of a graph, which is defined as given below.

\begin{definition}{\rm 
A {\em subdivision} of a graph $G$ is a graph obtained by introducing a new vertex to every edge of G (see \cite{CH,RD}).}
\end{definition}

In the following result, the curling number of a subdivision of a graph $G$ is determined.

\begin{proposition}
The curling number of a subdivision of a graph $G$ is $\epsilon+r$, where $r$ is the number of vertices of degree 2 in $G$ and $\epsilon$ is the number of edges of $G$.
\end{proposition}
\begin{proof}
Let $V'$ be the set of new vertices introduced to the edges of $G$.  Since a vertex is introduced to every edge of $G$, we have $|G|=\epsilon$, where $\epsilon$ is the number of edges of $G$.  For every verted $v'$ in $V', d(v')=2$, where $0\leq r \leq n$.  Therefore, the degree sequence of the subdivision graph $G$ is given by $(2)^{\epsilon+r}\circ S_0$, where $S_0$ is a subsequence of the degree sequence $S$ which is of the form $(a_1)^{r_1}\circ (a_2)^{r_2}\circ \ldots \circ (a_k)^{r_k}$ with $a_i\neq 2$.

\textbf{Case - 1 }: When $G$ is a tree.  Then $G$ has atleast 2 pendent vertices.  If all the internal vertices of $G$ have the same degree, then $r+r_1+r_2+\ldots +r_k=n$.  Therefore, $\sum_{i=1}^{k}r_i\geq n-2 < n-1 (=\epsilon)$.  Hence the number of elements in $S$ can have a power that is greter than $\epsilon$.  Therefore, $\epsilon+r$ is the highest power of elements in $S$.

\textbf{Case - 2} : If $G$ is not a tree.  Then we have $\epsilon \geq |V|(=n)$, that is $\epsilon+r$ is the highest power of an element in the sequence $S$ of the subdivision graph $G'$ of the given graph $G$.  Therefore, $cn(G')=\epsilon +r$.
\end{proof}

\subsection{Curling number of super subdivision of a graph}

Let us first recall the definition of the super subdivision of a given graph as follows.

\begin{definition}{\rm 
\cite{IS} Let $G$ be a graph with $n$ vertices and $\epsilon$ edges.  Then a \textit{super subdivision} $H$ of $G$ is a graph obtained by replacing every edge $e_i$ of $G$ by a complete bipartite graph $K_{2,m}$.  An \textit{arbitrary super subdivision} $G'$ of $G$ is a graph obtained by replacing every edge $e_i$ of $G$ by a complete bipartite graph $K_{2,m_i}$, where $1\leq i \leq \epsilon$.}
\end{definition}

The curling number of the super subdivision of a graph is determined in the following theorem.

\begin{theorem}
The curling number of an arbitrary super subdivision graph is $$cn(G')=\sum_{i=1}^{\epsilon}m_i$$.
\end{theorem}
\begin{proof}
Let $G$ be a graph on $n$ vertices and $\epsilon$ edges.  Also, let $S=(d_1, d_2, \ldots d_n)$ be the degree sequence of $G$.  Without loss of generality, let $S=(a_1)^{r_1}\circ (a_2)^{r_2}\circ \ldots \circ (a_k)^{r_k}$.  Let $G'$ be an arbitrary super subdivision of a graph $G$ obtained by replacing its edges by the complete bipartite graph $K_{2,m_i}$; $1\leq i \leq \epsilon$.  Clearly $V(G)\subset V(G')$. It can then be noted that the degree sequence of the vertices in $V(G)$ in $G'$ is given by $(m_1d_1, m_2d_2, \ldots, m_nd_n)$.  Let $S_1$ be the degree subsequence of the vertices of $G$ in $G'$. Therefore, $S_1=(\alpha_1)^{t_1} \circ (\alpha_2)^{t_2} \circ \ldots \circ (\alpha_l)^{t_l}$; where $\alpha_i$ takes the form $\sum m_sd_s$ for some values of $s\leq \epsilon$. Let $V' = V(G)'-V(G)$.  Then, every element $V'$ is of degree 2 and then, the degree sequence of the vertices in $V'$ is $2^{\sum m_i}, 1\leq i\leq \epsilon$. Therefore, the degree sequence of $G'=S_1\circ S = (2)^{\sum m_i}\circ (\alpha_1)^{t_1} \circ (\alpha_2)^{t_2}\circ \ldots (\alpha_k)^{t_l}$.  It is to be noted that each $t_j<\sum_{i=1}^{\epsilon}m_i$ and hence $cn(G')=\sum_{i=1}^{\epsilon}m_i$.
\end{proof}

In view of the above theorem, the curling number of a super subdivision of a graph $G$ can be determined as follows.

\begin{corollary}
The curling number of a super subdivision graph is $cn(G')=m\epsilon$ where $\epsilon$ is the size of $G$.
\end{corollary}
\begin{proof}
For a subdivision graph $H$ of $G$, each edge is replaced by a complete bipartite graph $K_{2,m}$.  That is $m_i=m$ for all $1\leq i\leq \epsilon$.  Therefore, by above theorem, we have $cn(H)=\sum_{i=1}^{\epsilon}m_i=\sum_{i=1}^{\epsilon}m=\epsilon.m$
\end{proof}

\subsection{Curling number of the shadow graph of a graph}

The shadow graph of a given graph $G$ can be defined as given below.

\begin{definition}{\rm 
\cite{CGT} The \textit{shadow graph} of a graph $G$ is obtained from $G$ by adding, for each vertex $v$ of $G$, a new vertex $v'$, called the shadow vertex of $v$, and joining $v'$ to the neighbours of $v$ in $G$.  The shadow graph of a graph $G$ is denoted by $S(G)$.} 
\end{definition}

The curling number of a shadow graph of a given graph $G$ is determined in the following theorem.

\begin{theorem}
Let $G$ be a graph and $G'$ be its shadow graph.  Then $cn(G')=cn(G)+\eta$, where $\eta$ is the number of vertices of $G$ having degree twice the most repeating degree in the degree sequence of $G$.  
\end{theorem}
\begin{proof}
Let $G$ be a graph on $n$ vertices and let $S=d_1^{n_1}\circ d_2^{n_2}\circ d_3^{n_3}\circ \ldots \circ d_r^{n_r}$, where $n_1+n_2+n_3+\ldots +n_r=n$. Without lose of generality, let $cn(G)=n_1$.  Now, consider the shadow graph $G'$ of $G$.  Let $U$ denotes the set of vertices in $G'$ which are common to $G$ also and $V$ be the set of newly introduced vertices in $G'$ which are not in $G$.  Therefore, the degree subsequence of $G'$ induced by the vertices in $U$ will be $S_1=(2d_1)^{n_1}\circ (2d_2)^{n_2}\circ (2d_3)^{n_3}\circ \ldots \circ (2d_r)^{n_r}$ and the degree subsequence of $G'$ induced by the vertices in $V$ will be $d_1^{n_1}\circ d_2^{n_2}\circ d_3^{n_3}\circ \ldots \circ d_r^{n_r}$.  If there exists some $d_j$ in $S_2$ (which will also be in $S$), such that $d_j=2d_1$ then $cn(G')=n_1+n_j=cn(G)+\eta$.
\end{proof}




\section{Conclusion}{\rm
In this paper, we have discussed the curling numbers of certain graph classes associated with certain given graphs.  There are several problems in this area which demands intense further investigations.  Some of the open problems we have identified during our study are the following.

\begin{problem}{\rm 
Determine the compound curling number of an arbitrary subdivision of a graph.}
\end{problem}

\begin{problem}{\rm 
Determine the compound curling number of an arbitrary super subdivision of a graph.}
\end{problem}

\begin{problem}{\rm
Determine the curling number and the compound curling number of the line graphs corresponding to other well known small graphs such as sun graphs, sunlet graphs, web graphs etc.}
\end{problem}

\begin{problem}{\rm
Determine the curling number and the compound curling number of the total graphs corresponding to various graph classes.}
\end{problem}

The concepts of curling number and compound curling number of certain graph powers and discussed certain properties of these new parameters for certain standard graphs. More problems regarding the curling number and compound curling number of certain other graph classes, graph operations, graph products and graph powers are still to be settled. All these facts highlights a wide scope for further studies in this area.

\end{document}